\documentclass[11pt]{amsart}
\usepackage{amsfonts,amssymb,amsmath}
\usepackage[margin=2.1 cm]{geometry}
\renewenvironment{proof}{{\bfseries Proof.}}{\qed}
\numberwithin{equation}{section} 
\theoremstyle{plain}
\newtheorem{theorem}{Theorem}[section] 
\newtheorem{pro}[theorem]{Proposition} 
\newtheorem{cor}[theorem]{Corollary} 
\newtheorem{lemma}[theorem]{Lemma} 
\theoremstyle{definition}
\newtheorem{defn}[theorem]{Definition} 
\newtheorem{remark}[theorem]{Remark} 
\newtheorem{Example}[theorem]{Example}

\newcommand{\C}{\mathbb{C}}

\newcommand{\thmref}[1]{Theorem~\ref{#1}}
\newcommand{\lemref}[1]{Lemma~\ref{#1}}

\numberwithin{equation}{section} 
 
\def \SL{ {\rm SL}}

\begin{document}
	\title[$c$-Reversibility in $\SL(n, \C)$] {Conjugate Reversibility in  Complex Special Linear Groups}
	\author[K. Gongopadhyay and R. Mondal] { Krishnendu Gongopadhyay \and Rahul Mondal}

\address{Indian Institute of Science Education and Research (IISER) Mohali,
		Knowledge City,  Sector 81, S.A.S. Nagar 140306, Punjab, India}
	\email{krishnendug@gmail.com, krishnendu@iisermohali.ac.in}

	\address{Indian Institute of Science Education and Research (IISER) Mohali,
		Knowledge City,  Sector 81, S.A.S. Nagar 140306, Punjab, India}
	\email{canvas.rahul@gmail.com}

\makeatletter
\@namedef{subjclassname@2020}{\textup{2020} Mathematics Subject Classification}
\makeatother


    
    \date{\today}

      \begin{abstract}
        We introduce and study conjugate reversibility (or $c$-reversibility) in the complex special linear group $\SL(n,\C)$  where an element is conjugate to the inverse of its complex conjugate. We prove that in $\SL(n, \C)$, every $c$-reversible element is strongly $c$-reversible. We provide a complete classification of $c$-reversible elements based on their conjugacy invariants.  This leads to an algebraic characterization of projective transformations. As a special case, a finer classification in $\SL(4, \C)$ is obtained in terms of trace conditions and resultant computations. 
    \end{abstract}
  	\maketitle

	\section{Introduction}\label{sec:intro}
    	Among the many ways to express symmetry, one particularly intriguing notion is that of reversibility, the ability of a transformation to be conjugated to its own inverse. Reversibility appears in many areas across group theory, geometry, and dynamical systems, revealing deep patterns in how structures behave under time-reversal. For works related to reversibility, see, for example, \cite{BR}, \cite{dhs},  \cite{Dj}, \cite{El}, \cite{GM},  \cite{OS}, \cite{sho}, \cite{st}.
    
    In this article, we explore a new perspective on reversibility within the framework of the special complex linear groups, $\mathrm{SL}(n, \mathbb{C})$, the group of $n \times n$ complex matrices with determinant one. While the classification of reversible elements in $\mathrm{SL}(n, \mathbb{C})$ can be derived from the corresponding classification in the general linear group, it is noteworthy that the classification of products of two involutions, or equivalently, of  strongly reversible elements in $\mathrm{SL}(n, \mathbb{C})$, has been obtained only recently in  \cite{GLM2}. 
    
    In this work, we introduce and study the concept of \emph{conjugate reversibility}, or $c$-reversibility, where an element is conjugated to the inverse of its complex conjugate. This setting naturally intertwines group symmetries with complex conjugate structures and leads to new insights into the classification of complex projective transformations.

\begin{defn}
An element $g$ in $\SL(n, \C)$ is called \emph{conjugate reversible} or \emph{c-reversible} if there exists an element $h$ in $\SL(n, \C)$ such that $hgh^{-1}$ = $\overline{g}^{-1}$. 
 If further $h \overline h$ = $1$, then $g$ is called \emph{strongly $c$-reversible} and $h$ is called an \emph{involutory $c$-symmetry}.
	\end{defn}
This notion is a special case of the twisted reversibility introduced in \cite{gms}, \cite{dc}. To understand and classify $c$-reversible elements, we prove a characterization of $c$-reversibility based on Jordan canonical forms. We find that c-reversibility is closely linked to the way eigenvalues pair with their conjugate reciprocals. Further we prove the following theorem.
\begin{theorem}\label{2.8}
		$A \in \SL(n,\mathbb{C})$ is $c$-reversible if and only if it is strongly $c$-reversible.
	\end{theorem}
    The group $\mathrm{SL}(n, \mathbb{C})$ acts on the complex projective space $\mathbb{C}\mathrm{P}^n$ by projective automorphisms. The projective special linear group $\mathrm{PSL}(n, \mathbb{C})$,  the quotient of $\mathrm{SL}(n, \mathbb{C})$ by its center, can be naturally identified with the full group of holomorphic projective automorphisms of $\mathbb{C}\mathrm{P}^n$.  Discrete groups of projective automorphisms arise naturally as monodromy groups of differential equations and play a central role in the study of geometric structures. The dynamics of these automorphisms are of particular interest, as they generalize the classical dynamics of projective transformations of the Riemann sphere.  In recent years, there has been growing interest in the dynamical properties of projective automorphisms, particularly in connection with the study of Kulkarni limit sets associated with complex projective automorphisms. This topic has been explored extensively in the work of Seade, Verjovsky, and others; see \cite{cns} for a comprehensive survey.
    
A key step in understanding the dynamics of individual transformations is to obtain an algebraic characterization that reflects their dynamical behavior. In the classical case of $\mathrm{SL}(2, \mathbb{C})$, elements are divided into three mutually exclusive classes:  elliptic, parabolic, and hyperbolic, based on their action on the Riemann sphere, with the absolute value of the trace providing a complete algebraic invariant distinguishing these classes.

In higher dimensions, one can still divide the automorphisms into the four classes as follows. 

\begin{defn}	\label{def:classification}
		Let $g$ be an element in $\SL(n+1, \C)$ that is acting as a projective automorphism. 
		\begin{enumerate}

			\item $g$ is called \emph{elliptic} if it is semisimple and all of its eigenvalues lie on the unit circle, that is, they have modulus one.

            \item An element \( g \) is called \emph{parabolic} if it is not semisimple, and all of its eigenvalues have modulus one.

\item An element \( g \) is called \emph{loxodromic} if it is semisimple and has at least one eigenvalue whose modulus is not equal to one.

\item An element \( g \) is called \emph{loxoparabolic} if it is not semisimple and has at least one eigenvalue whose modulus is not equal to one.
\end{enumerate}
        \end{defn}

        A matrix $g$ in $\SL(n, \C)$ is called \emph{$k$-loxodromic} if it has $k$ pairs of eigenvalues $re^{i \theta_j}$, $r^{-1}e^{ \theta_j}$ with $r_j \neq 1$ for $j=1, \ldots, k$, and all other eigenvalues are unit modulus complex numbers. We adopt the
convention of taking $k\geq 0$ with the understanding that a $0$-loxodromic means that
all eigenvalues are unit modulus complex numbers. An element $g$ in $\SL(n, \C)$ is \emph{regular} if it has no repeated roots. It follows that a regular $k$-loxodromic is a semisimple loxodrimic element. 

A polynomial $f(x)$ is called $c$-reciprocal when for every root  $\lambda$, it has a root $\overline \lambda^{-1}$ with the same multiplicity.

\medskip 
It is desirable to have an algebraic characterization of these classes using conjugacy invariants.  However, the situation becomes considerably more intricate with larger $n$. For the special subgroup ${\rm SU}(n,1) \subset \mathrm{SL}(n, \mathbb{C})$, whose elements act as isometries of complex hyperbolic space, algebraic characterizations of dynamical types have been established, see \cite{wmg}, \cite{go1}, and \cite{gpp}. Nonetheless, for the full group $\mathrm{SL}(n, \mathbb{C})$, obtaining such characterizations remains a challenging.   Navarette \cite{n} extended Goldman's classification in ${\rm SU}(2,1)$ to $\SL(3, \C)$. In \cite{clu}, Cano et al. attempted to develop such a classification in arbitrary dimensions. However, their approach led to a classification that is geometric and dynamic in nature rather than purely algebraic. Their method relies on an analysis of equicontinuity domains and Kulkarni limit sets associated with cyclic subgroups. Through this perspective, Cano et al. provided a geometric characterization of the automorphisms.

As an application of our classification of $c$-reversibility, we offer an algebraic characterization for the projective automorphisms. Note that an element $g$ in ${\rm PSL}(n, \C)$ is $c$-reversible if one of the lifts $\tilde g$ in $\SL(n, \C)$ is $c$-reversible. It is easy to see that if $g$ is $c$-reversible, then there is always a lift $\tilde g$ that is $c$-reversible. So, to understand $c$-reversibility in ${\rm PSL}(n, \C)$, it is enough to consider $c$-reversibility in $\SL(n, \C)$. Now, the automorphisms may be divided into two classes: $c$-reversible and those which are not $c$-reversible. These classes can be distinguished from each other purely in terms of their conjugacy invariants by the following theorem. 

\begin{theorem}\label{CM} 
An element $A \in \SL(n, \C)$ is $c$-reversible if and only if the minimal and characteristic polynomials of $A$ are  $c$-reciprocal. 
\end{theorem}
An immediate consequence of this theorem is that the only elements that are not $c$-reciprocal fall within the loxoparabolic class. Moreover, 
$c$-reciprocal elements can be algebraically characterized by their characteristic polynomial, as given below. We use the following terminology to state this characterization.  
\begin{theorem}\label{RT}
Let $A \in \SL(n, \C)$ such that $A$ is $c$-reciprocal. Suppose the characteristic polynomial of \( A \) is 
\[
x^n - c_{n-1}x^{n-1} + c_{n-2}x^{n-2} - \dots + (-1)^n = 0.
\]

 Let \( R(\chi_A, \chi_A') \) denote the resultant of the characteristic polynomial \( \chi_A \) and its derivative \( \chi_A' \).  Then:
\begin{itemize}
    \item[(i)] \( A \) is \emph{regular \( 2m \)-loxodromic} if and only if
    \[
    R(\chi_A, \chi_A') > 0.
    \]

    \item[(ii)] \( A \) is \emph{regular \( (2m + 1) \)-loxodromic} if and only if
    \[
    R(\chi_A, \chi_A') < 0.
    \]

    \item[(iii)] \( A \) is \emph{not regular} (i.e., has a repeated eigenvalue) if and only if
    \[
    R(\chi_A, \chi_A') = 0.
    \]
\end{itemize}
 \end{theorem}
We must note that the terminology and the proof of this theorem essentially follow from the ideas of \cite[Theorem 3.1]{gpp}. However, we shall give the details for completeness.

 \medskip 
 
In the special case of $n=4$, that is for $\SL(4, \C)$,  we further provide a characterization to characterize the $c$-reversibility of an element $A$ in terms of the coefficients of the characteristic polynomial. 

\begin{theorem}\label{CT}
Let \( A \in \mathrm{SL}(4, \mathbb{C}) \) have characteristic polynomial
\[
\chi_A(x) = x^4 - c_3 x^3 + c_2 x^2 - c_1 x + 1.
\]
If at least one of the following conditions holds, then \( A \) is not \( c \)-reciprocal:
\begin{enumerate}
    \item[(i)] \( c_3 \ne \overline{c_1} \),
    \item[(ii)] \( c_2 \notin \mathbb{R} \).
\end{enumerate}

\medskip

Suppose \( c_3 = \overline{c_1} \) and \( c_2 \in \mathbb{R} \). Then the following results hold:
\begin{enumerate}
    \item If the trace sequence \( \{\operatorname{tr}(A^n)\}_{n \in \mathbb{N}} \) is bounded, then \( A \) is \( c \)-reversible.
    
    \item If the trace sequence is unbounded and the minimal polynomial of \( A \) has degree different from 3, then \( A \) is still \( c \)-reversible.
    
    \item
    \begin{enumerate} \item  If the minimal polynomial has degree 3 and \( \operatorname{tr}(A) = x + i y \) with \( x, y \ne 0 \), then \( A \) remains \( c \)-reversible.
    
    \item If \( \operatorname{tr}(A) \in \mathbb{R} \), then \( A \) is \( c \)-reversible if and only if
    \[
    c_2 \ne \frac{c_3^2}{4} + 2.
    \]
    And if \( \operatorname{tr}(A) \) is purely imaginary, then \( A \) is \( c \)-reversible if and only if
    \[
    c_2 \ne \frac{c_3^2}{4} - 2.
    \]
\end{enumerate}
\end{enumerate}
\end{theorem}

\subsection*{Structure of the paper}
After providing an introduction in Section \ref{sec:intro}, we discuss some preliminaries in Section \ref{sec:prelim}. In Section~\ref{sec:SCR}, we prove that every $c$-reversible element is strongly $c$-reversible. In section,~\ref{sec:AC} we characterize $c$-reversible elements using the resultant of their characteristic polynomials, and also describe $SL(4,\mathbb{C})$ in terms of trace conditions.

\section{Preliminaries}\label{sec:prelim}
In this section we observe some basic observations about $c$-reversibility. 
	\begin{lemma}
	The $c$-reversible and strongly $c$-reversible are conjugacy invariants in $\SL(n, \C)$ .
\end{lemma}
\begin{proof}
	Suppose $g$ is $c$-reversible, then there exists an element $h \in \SL(n,\C)$ such that, $hgh^{-1} = \bar{g}^{-1}$. Now for any $k \in \SL(n,\C)$, we have:
	\[
		(\bar{k}hk^{-1})~(kgk^{-1})~(\bar{k}hk^{-1})^{-1} 
	\]
	\[
		=\bar{k}hk^{-1}kgk^{-1}kh^{-1}\bar{k}^{-1}
	\]
	\[
		=\bar{k}\bar{g}^{-1}\bar{k}^{-1} = \overline{(kgk^{-1})}^{-1}.
	\]
	So, $kgk^{-1}$ is $c$-reversible.

	If $h\bar{h} = e$, then $(\bar{k}hk^{-1})\overline{(\bar{k}hk^{-1})} = h\bar{h} = e$. Therefore, if $g$ is strongly $c$-reversible, then $kgk^{-1}$ is strongly $c$-reversible. 
\end{proof}
	\begin{lemma}
	An element $g$ in $\SL(n, \C)$ is strongly $c$-reversible if and only if $g$ is a product of two involutory $c$-symmetries.
\end{lemma}
\begin{proof}
	Suppose $g$ is strongly $c$-reversible. Then there exists an element $h \in \SL(n, \C)$ such that 
	\[
	hgh^{-1} = \bar{g}^{-1} \quad \text{with} \quad h\bar{h} = e.
	\]
	\(\implies\) \(g = (h^{-1})(\bar{g}^{-1}h)\). Since $h$ is an involutary $c$-symmetry, $h^{-1}$ is also an involutary $c$-symmetry. Also, $\bar{g}^{-1}h$ is an involutary $c$-symmetry, because
	\[
	(\bar{g}^{-1}h)\overline{(\bar{g}^{-1}h)} = \bar{g}^{-1}hg^{-1}\bar{h} = h\bar{h} = e.
	\]
	Conversely, suppose $g$ is a product of two involutory $c$-symmetries. That is, suppose $g = ab$, where $a\bar{a} = e$ and $b\bar{b} = e$. Then 
	\[
	a^{-1}ga = ba = \bar{g}^{-1}.
	\]
	Thus, $g$ is strongly $c$-reversible.
\end{proof}
    \\
   
    Let 
\[
Z(g) = \{ h \in \SL(n, \C) \mid gh = hg \}
\]
be the centralizer of \( g \), and
\[
R(g) = \{ h \in G \mid hgh^{-1} = \bar{g}^{-1} \}
\]
be the set of \( c \)-reversers of \( g \).

Define:
\[
\mathcal{E}(g) = Z(g) \cup R(g).
\]

 It is easy to see that \( \mathcal{E}(g) \) is a subgroup of \( G \).

Observe that if \( h_1, h_2 \in R(g) \), then
\[
h_1 g h_1^{-1} = \bar{g}^{-1} = h_2 g h_2^{-1}
\Rightarrow h_2 h_1^{-1} g (h_2 h_1^{-1})^{-1} = g
\Rightarrow h_2 h_1^{-1} \in Z(g).
\]
\begin{lemma}
	$Z(g)$ is a normal subgroup of the extended centralizer of $\mathcal{E}(g)$ of index atmost two.
	
\end{lemma}
\begin{proof}
	Define a map $\phi$ : $\mathcal{E}(g) \to \mathbb{Z}_{2}$ by,
	\[
	\displaystyle \phi(h) = \begin{cases} 
		1 & \text{if $hgh^{-1}=g$} \\  
		-1 & \text{if $hgh^{-1}=\bar{g}^{-1}$}
	\end{cases} \]
\\ This is a homomorphrism with kernel $Z(g)$.
\end{proof}
\begin{remark} Although our focus is $\SL(n, \C)$, similar observations should carry over for $\phi$-reversibility in any group $G$, where $\phi: G \to G$ is an anti-automorphism of $G$, compare \cite{dc}.
\end{remark} 

\section{$c$-Reversibility in ${\rm SL}(n, \C)$}\label{sec:SCR}

	\begin{pro}\label{JB}
		Let $A$ $\in$ ${\rm SL}(n,\mathbb{ C})$ with jordan cannonical form consisiting of Jordan block $J(\lambda_1,n_1)$, $J(\lambda_2,n_2), \ldots,  J(\lambda_k,n_k)$ where $\lambda_i$ $\neq$ 0, is c-reversible if and only if the block can be partioned into pairs \{ $J(\lambda,m)$, $J(\overline{\lambda}^{-1},m)$\} or, singeltons \{$J(\mu,m)$\} where $\lambda,\mu \in \mathbb{C}\setminus \{0\}$ and $|{\lambda}| \neq 1$, $|\mu|=1$.
	\end{pro}
	\begin{proof}
		The inverse of $\overline{J(\lambda,n)}$ where $\lambda \neq 0$ is an $n \times n$ upper triangular matrix with entries $\overline{\lambda}^{-1}$ on the diagonal. Since the geometric multiplicity of $J(\lambda,n)$ is 1. We can deduce that the Jordan normal form $\overline{J(\lambda,n)}^{-1}$ is a single block $J(\overline{\lambda}^{-1},n)$. More generally, if $A \in \SL(n,\mathbb{C})$ is in Jordan normal form with blocks $J(\lambda_{1},n_{1}),\ldots, J(\lambda_{k},n_{k})$,  then the Jordan normal form of $\bar{A}^{-1}$ has blocks $J(\overline \lambda_{1}^{-1},n_{1}),\ldots, J(\overline \lambda_{k}^{-1},n_{k})$.
		
		Since the Jordan canonical form is unique up to permutation, So $A$ is conjugate to $\overline{A}^{-1}$ if and only if
		$$\{J(\lambda_{1},n_{1}),\ldots,J(\lambda_{k},n_{k})\} = \{J(\overline \lambda_{1}^{-1},n_{1}),\ldots, J(\overline \lambda_{k}^{-1},n_{k})\}.$$
		This completes the proof. 
	\end{proof}
    \\
    \subsection*{Proof of \thmref{CM}}

Suppose \( A \in \mathrm{SL}(n, \mathbb{C}) \).

First assume \( A \) is \( c \)-reversible. By Proposition~\ref{JB}, the Jordan canonical form of \( A \) consists of Jordan blocks that either appear in \( c \)-reciprocal pairs
  \[
  \{ J(\lambda, m),\, J(\overline{\lambda}^{-1}, m) \}
  \quad \text{with } |\lambda| \ne 1,
  \]
  or as singleton blocks
  \[
  \{ J(\mu, m) \}
  \quad \text{with } |\mu| = 1.
  \]
  Thus, the eigenvalues of \( A \) are closed under the map \( \lambda \mapsto \overline{\lambda}^{-1} \), and the Jordan structure is preserved under this symmetry. It follows that both the characteristic and minimal polynomials of \( A \) are \( c \)-reciprocal.

  \medskip Conversely, suppose the characteristic and minimal polynomials of \( A \) are \( c \)-reciprocal. Then the sets of eigenvalues and elementary divisors are closed under \( \lambda \mapsto \overline{\lambda}^{-1} \). Consequently, the Jordan blocks in the Jordan canonical form of \( A \) must occur in the pattern described in Proposition~\ref{JB}, either as
  \begin{itemize}
      \item pairs \( \{ J(\lambda, m), J(\overline{\lambda}^{-1}, m) \} \) with \( |\lambda| \ne 1 \), or
      \item singletons \( \{ J(\mu, m) \} \) with \( |\mu| = 1 \).
  \end{itemize}
  By the proposition ~\ref{JB}, this implies that \( A \) is \( c \)-reversible. \qed
  \\
  \begin{cor}\label{EPL}
      Let $A\in \SL(n,\mathbb{ C})$ such that A is either elliptic, parabolic or loxodromic then $A$ is $c$-reversible.
  \end{cor}

	\begin{lemma}\label{BL1} Let $n>1$. Assume $r <n$ then the following holds. 
		$$\binom{n-1}{r}-\binom{n}{r+1}+\binom{n}{r+2}+....+(-1)^{n-r}\binom{n}{n} = 0.$$
		
	\end{lemma}
	\begin{proof}
		We proceed by induction on $r$. For $r=0$ that is, 	\[\binom{n-1}{0}-\binom{n}{1}+\binom{n}{2}+....+(-1)^{n}\binom{n}{n} \]
		\[=\binom{n}{0}-\binom{n}{1}+\binom{n}{2}+....+(-1)^{n}\binom{n}{n} =0\]
        which is a well-known identity (alternating sum of binomial coefficients in a row of Pascal's triangle equals zero for $n\geq1)$. Thus the base case holds.

		Assume that the identity hold for $r=k$. 	We must show it holds for \( r = k+1 \):
\[
\binom{n-1}{k+1} - \binom{n}{k+2} + \binom{n}{k+3} - \cdots + (-1)^{n-k-1} \binom{n}{n} = 0.
\]

Using the identity \( \binom{n-1}{k+1} = \binom{n}{k+1} - \binom{n-1}{k} \), we rewrite:
\[
\binom{n-1}{k+1} = \binom{n}{k+1} - \binom{n-1}{k}.
\]
Substituting this in:
\[
\bigg (\binom{n}{k+1} - \binom{n-1}{k}\bigg) - \binom{n}{k+2} + \binom{n}{k+3} - \cdots + (-1)^{n-k-1} \binom{n}{n}.
\]
Group terms:
\[
\left[ \binom{n}{k+1} - \binom{n}{k+2} + \binom{n}{k+3} - \cdots + (-1)^{n-k-1} \binom{n}{n} \right] - \binom{n-1}{k}.
\]
By the inductive hypothesis, the bracketed expression equals \( \binom{n-1}{k} \), so the total becomes:
\[
\binom{n-1}{k} - \binom{n-1}{k} = 0.
\]

This completes the inductive step and hence the proof.
	\end{proof}
	\begin{lemma}\label{BL2}
Assume \( k < n \). Then the following identity holds:
\[
\binom{k}{k}\binom{n}{k} - \binom{k+1}{k}\binom{n}{k+1} + \binom{k+2}{k}\binom{n}{k+2} - \cdots + (-1)^{n-k} \binom{n}{k} \binom{n}{n} = 0.
\]
\end{lemma}

\begin{proof}
We use the identity:
\[
\binom{j}{k} \binom{n}{j} = \binom{n}{k} \binom{n-k}{j-k}
\quad \text{for } k \leq j \leq n.
\]
This allows us to rewrite each term in the sum as:
\[
\binom{j}{k} \binom{n}{j} = \binom{n}{k} \binom{n-k}{j-k}.
\]

Hence, the entire expression becomes:
\[
\sum_{j = k}^{n} (-1)^{j-k} \binom{j}{k} \binom{n}{j}
= \binom{n}{k} \sum_{i=0}^{n-k} (-1)^i \binom{n-k}{i}.
\]

This is the alternating sum of the binomial coefficients in the \( (n-k) \)-th row of Pascal's triangle, which evaluates to:
\[
\sum_{i=0}^{n-k} (-1)^i \binom{n-k}{i} = (1 - 1)^{n-k} = 0.
\]

Thus, the full expression simplifies to the value $0$. 
This completes the proof.
\end{proof}

\begin{lemma}\label{c-B}
	Define a matrix $B = [b_{i,j}]_{1 \leq i,j \leq n}$, where
\[
b_{i,j} = 
\begin{cases} 
0 & \text{if } j < i, \\  
(-1)^{j+1} \binom{j-2}{i-2} \dfrac{b}{\lambda^{i+j-2}} & \text{ if } j \geq i ~\text{ and }~ i > 1, \\  
b & \text{if } i = j = 1, \\
0 & \text{if } i = 1 ~\text{ and }~ 2 \leq j \leq n,
\end{cases}
\]
where \( |b| = 1 \) and \( |\lambda| = 1 \).

The matrix $B$ defined above satisfies $B\overline{B} = I$.
\end{lemma}

\begin{proof}
Since $B$ is upper triangular, so is $B\overline{B}$.

\medskip
\noindent
\textit{Diagonal entries:} For $i = j$,
\[
[B\overline{B}]_{i,i} = b_{i,i} \cdot \overline{b_{i,i}} = \left| \frac{b}{\lambda^{2i-2}} \right|^2 = \frac{|b|^2}{|\lambda|^{4i - 4}} = 1,
\]
since $|b| = 1$ and $|\lambda| = 1$.

\medskip

\noindent
\textit{Off-diagonal entries:} Let $j > i$. Then,
\[
[B\overline{B}]_{i,j} = \sum_{k=0}^{j-i} b_{i,i+k} \cdot \overline{b_{i+k,j}}.
\]
Each term is:
\[
b_{i,i+k} \cdot \overline{b_{i+k,j}} = 
(-1)^{i+k+1} \binom{i+k-2}{i-2} \frac{b}{\lambda^{2i+k-2}} \cdot 
(-1)^{j+1} \binom{j-2}{i+k-2} \frac{\bar{b}}{\bar{\lambda}^{i+k+j-2}}.
\]
\[
= (-1)^{i+j+k+2} \binom{i+k-2}{i-2} \binom{j-2}{i+k-2} 
\cdot \frac{|b|^2}{\lambda^{2i-2} \bar{\lambda}^{i+j-2}}.
\]
So
\[
[B\overline{B}]_{i,j} = 
(-1)^{i+j+2} \frac{1}{\lambda^{2i-2} \bar{\lambda}^{i+j-2}} 
\sum_{k=0}^{j-i} (-1)^k \binom{i+k-2}{i-2} \binom{j-2}{i+k-2}.
\]
The sum vanishes by \lemref{BL2} so $[B\overline{B}]_{i,j} = 0$ for all $j > i$.
\end{proof}

\begin{Example}
For $n = 4$, the matrix $B$ is
\[
B =
\begin{bmatrix}
b & 0 & 0 & 0 \\
0 & \dfrac{-b}{\lambda^2} & \dfrac{b}{\lambda^3} & \dfrac{-b}{\lambda^4} \\
0 & 0 & \dfrac{b}{\lambda^4} & \dfrac{-2b}{\lambda^5} \\
0 & 0 & 0 & \dfrac{-b}{\lambda^6}
\end{bmatrix}.
\]
\end{Example}

	\begin{lemma}
Let \( A \in \mathrm{SL}(n,\mathbb{C}) \) be a Jordan block of the form \( J(\lambda,n) \) with \( |\lambda| = 1 \). Then the matrix \( B \) defined above is a strongly \( c \)-reverser of \( A \).
\end{lemma}

\begin{proof}
We aim to show that \( BAB^{-1} = \overline{A}^{-1} \), i.e., \( BA = \overline{A}^{-1}B \).

Since both \( A \) and \( B \) are upper triangular, the product \( BA \) and \( \overline{A}^{-1}B \) are also upper triangular. Therefore, for \( i > j \), we have \( [BA]_{i,j} = [\overline{A}^{-1}B]_{i,j} = 0 \).

Now assume \( i \leq j \). Then:
\[
[BA]_{i,j} = b_{i,j-1} + \lambda b_{i,j},
\]
\[
[\overline{A}^{-1}B]_{i,j} = \sum_{k=i}^j (-1)^{k-i} \lambda^{k-i+1} b_{k,j}.
\]

So we need to verify:
\[
b_{i,j-1} = -\lambda^2 b_{i+1,j} + \lambda^3 b_{i+2,j} - \cdots + (-1)^{j-i} \lambda^{j-i+1} b_{j,j}.
\]

Substituting the formula for \( b_{k,j} \), both sides simplify to:
\[
(-1)^j \binom{j-3}{i-2} \frac{b}{\lambda^{i+j-3}} = 
(-1)^{j+1} \frac{b}{\lambda^{i+j-3}} \sum_{k=0}^{j-i} (-1)^k \binom{j-2}{i-2+k}.
\]

This simplifies to:
\[
(-1)^j \frac{b}{\lambda^{i+j-3}} \left[ \binom{j-3}{i-2} - \binom{j-2}{i-2} + \binom{j-2}{i} - \cdots + (-1)^{j-i} \binom{j-2}{j-2} \right] = 0,
\]
which is zero.  (see Lemma~\ref{BL1}) when setting \( j - 2 = n \), \( i - 2 = r \).

Since
\[
\det(B) = 
\begin{cases}
\displaystyle \frac{b^{n}}{\lambda^{n(n-1)}} & \text{if } n \equiv 0,1 \pmod{4}, \\[8pt]
\displaystyle -\frac{b^{n}}{\lambda^{n(n-1)}} & \text{otherwise}.
\end{cases}
\]
We can choose \( b \in \mathbb{C} \) such that \( |b| = 1 \) and \( \det B = 1 \) as \( |\lambda| = 1 \), this is always possible.

Also, from earlier, we have \( B \overline{B} = I \), so \( B \) is involutary $c$-symmetry and hence a strongly \( c \)-reverser of \( A \).
\end{proof}

Define another matrix $B(\lambda,n)$ by
\[
b_{i,j} = 
\begin{cases}
0 & \text{if } j < i, \\
(-1)^{j+1} \binom{j-2}{i-2} b\, \bar{\lambda}^{i+j-2} & \text{if } j \geq i,\ i>1, \\
b & \text{if } i = j = 1, \\
0 & \text{if } i = 1,\ j \geq 2,
\end{cases}
\]
where $|b|=1$.

We can verify similarly that
\[
B(\lambda,n) J(\bar{\lambda}^{-1},n) B(\lambda,n)^{-1} = \overline{J(\lambda,n)}^{-1}.
\]
\subsection*{Proof of \thmref{2.8}}
If $A$ contains a Jordan block of the form $J(\lambda,n)$ with $|\lambda|=1$, the result follows directly from the earlier construction.

Now suppose $A$ has a pair of Jordan blocks: $J(\lambda,n)$ and $J(\overline{\lambda^{-1}},n)$ with $|\lambda|\ne 0,1$. From the identity above, we get
\[
\overline{B(\lambda,n)}^{-1} J(\lambda,n) \overline{B(\lambda,n)} = \overline{J(\bar{\lambda}^{-1},n)}^{-1}.
\]

Define the matrix
\[
C = \begin{bmatrix}
0 & B(\lambda,n) \\
\overline{B(\lambda,n)}^{-1} & 0
\end{bmatrix}.
\]
Then we have
\[
C \begin{bmatrix}
J(\lambda,n) & 0 \\
0 & J(\bar{\lambda}^{-1},n)
\end{bmatrix} C^{-1} = 
\begin{bmatrix}
\overline{J(\lambda,n)}^{-1} & 0 \\
0 & \overline{J(\bar{\lambda}^{-1},n)}^{-1}
\end{bmatrix}.
\]

Next, note that
\[
\det B(\lambda,n) = 
\begin{cases}
b^n \bar{\lambda}^{n(n-1)} & \text{if } n \equiv 0,1 \mod 4, \\
-b^n \bar{\lambda}^{n(n-1)} & \text{otherwise}.
\end{cases}
\]
We can choose $b\in \mathbb{C}$ such that $b^n \bar{\lambda}^{n(n-1)}$ is real if $n$ is even, and purely imaginary if $n$ is odd, to ensure
\[
\det(C) = (-1)^n \frac{\det B(\lambda,n)}{\det \overline{B(\lambda,n)}} = 1.
\]

Finally, we check that
\[
C \cdot \overline{C} = 
\begin{bmatrix}
0 & B \\
\bar{B}^{-1} & 0
\end{bmatrix}
\begin{bmatrix}
0 & \bar{B} \\
B^{-1} & 0
\end{bmatrix}
= \begin{bmatrix}
I & 0 \\
0 & I
\end{bmatrix}. 
\]
So the matrix 
\[
\begin{bmatrix}
J(\lambda,n) & 0 \\
0 & J(\bar{\lambda}^{-1},n)
\end{bmatrix}
\]
is strongly $c$-reversible. 

Since every matrix, up to conjugacy,  consists of Jordan blocks of the above two types, the theorem follows. 
\qed
	\section{Algebraic Characterization}\label{sec:AC}
    \subsection*{Proof of \thmref{RT}}
Suppose $A$ is $c$-reversible such that it is regular $r$-loxodromic.  It has $r$ pairs of non-unit eigenvalues of the form $\lambda_j=e^{\ell_j + i\theta_j}$ and $\bar\lambda_j=e^{-\ell_j + i\theta_j}$, with $\ell_j > 0$, $s$ eigenvalues on the unit circle of the form $\mu_k=e^{i\phi_k}$, for $k = 1, \dots, s$, where  $n = 2r + s$.

Then the resultant of \( \chi_A \) and its derivative \( \chi_A' \) is given by:

\[
R(\chi_A, \chi_A') = (-1)^{\frac{n(n-1)}{2}}
\prod_{j=1}^r (\lambda_j - \bar\lambda_j^{-1})^2
\prod_{1 \leq j < k \leq r} 
(\lambda_j - \lambda_k)^2
(\bar\lambda_j^{-1} - \bar\lambda_k^{-1})^2
(\lambda_j - \bar\lambda_k^{-1})^2
(\bar\lambda_j^{-1} - \lambda_k)^2
\]
\[
\times
\prod_{j=1}^r \prod_{k=1}^s 
(\lambda_j - \mu_k)^2
(\bar\lambda_j^{-1} - \mu_k)^2
\prod_{1 \leq j < k \leq s} 
(\mu_j - \mu_k)^2
\]

We classify the pairwise eigenvalue differences into the following types:
\begin{itemize}
    \item Between $e^{\ell_j + i\theta_j}$ and $e^{-\ell_j + i\theta_j}$:
    \[
    (e^{\ell_j + i\theta_j} - e^{-\ell_j + i\theta_j})^2 =e^{2i\theta_j} \cdot 4\sinh^2(\ell_j).
    \]
    \item Between  $e^{\ell_j + i\theta_j}$ and $e^{-\ell_k + i\theta_k}$:
    \[
    (e^{\ell_j + i\theta_j} - e^{\ell_k + i\theta_k})^2(e^{-\ell_j + i\theta_j} - e^{-\ell_k + i\theta_k})^2 = e^{2i\theta_j+2i\theta_k}(2\cosh(\ell_j-\ell_k)-2\cos(\theta_j-\theta_k))^2 .
    \]
    Also,
     \[
    (e^{\ell_j + i\theta_j} - e^{-\ell_k + i\theta_k})^2(e^{-\ell_j + i\theta_j} - e^{\ell_k + i\theta_k})^2 = e^{2i\theta_j+2i\theta_k}(2\cosh(\ell_j+\ell_k)-2\cos(\theta_j-\theta_k))^2 .
    \]
    
    \item Between $e^{\ell_j + i\theta_j}$ and $e^{i\phi_k}$:
    \[
    (e^{\ell_j + i\theta_j} - e^{i\phi_k})^2(e^{-\ell_j + i\theta_j} - e^{i\phi_k})^2 = e^{2i\theta_j+2i\phi_k}(2\cosh(\ell_j)-2\cos(\theta_j-\phi_k))^2.
    \]
    \item Between  $e^{i\phi_j}$ and $e^{i\phi_k}$:
    \[
    (e^{i\phi_j} - e^{i\phi_{k}})^2 = -e^{i\phi_j+i\phi_k}(2-2\cos(\theta_j-\theta_k)).
    \]
\end{itemize}

Hence the full product is :
\[
R(\chi_A, \chi_A') = (-1)^{\frac{n(n-1)}{2}}\prod_{j=1}^r e^{2i\theta_j} \prod_{1 \leq j < k \leq r} e^{4i\theta_j+4i\theta_k}\prod_{j=1}^r \prod_{k=1}^s e^{2i\theta_j+2i\phi_k}\cdot \prod_{1 \leq j < k \leq s} -e^{i\phi_j+i\phi_k}  \times (\text{positive real product}).
\]
\\
\[
=(-1)^{\frac{n(n-1)}{2}}(-1)^{\frac{s(s-1)}{2}}\prod_{j=1}^r e^{2i\theta_j} \prod_{j=1}^r e^{4i(r-1)\theta_j}\prod_{j=1}^s e^{i(s-1)\phi_j} \prod_{j=1}^r e^ {2is\theta_j}\prod_{j=1}^s e^ {2is\phi_j}  \times (\text{positive real product}).
\]
\\
\[
=(-1)^{\frac{n(n-1)}{2}}(-1)^{\frac{s(s-1)}{2}}\prod_{j=1}^r e^{2i\theta_j(2r+s-1)} \prod_{j=1}^s e^{i\phi_j(2r+s-1)} \times(\text{positive real product}).
\]
Recall that
\[
\prod_{j=1}^r e^{2i\theta_j(n-1)} \prod_{j=1}^s e^{i\phi_j(n-1)} = (\det A)^{n-1} = 1,
\]
since \( \det A = 1 \). Thus:
\[
R(\chi_A, \chi_A') = (-1)^{\frac{n(n-1)}{2}} (-1)^{\frac{s(s-1)}{2}} \times (\text{positive real product}).
\]

Now, compute the parity of the exponent:
\[
(-1)^{\frac{n(n-1)}{2}} (-1)^{\frac{s(s-1)}{2}} =(-1)^{\frac{(2r+s)(2r+s-1)+s(s-1)}{2}} = (-1)^{\frac{2s^2-2s-2r}{2}} =(-1)^r,
\]

Thus:
\begin{itemize}
    \item If $r$ is even, $R(\chi_A, \chi_A') > 0$ and $A$ is regular $2m$-loxodromic.
    \item If $r$ is odd, $R(\chi_A, \chi_A') < 0$ and $A$ is regular $(2m+1)$-loxodromic.
    \item If $\chi_A$ has repeated roots, then $R(\chi_A, \chi_A') = 0$, and $A$ is not regular.
\end{itemize}

This proves the theorem.
\qed

\medskip

	\subsection{Proof of \thmref{CT}}
    If $c_3\neq \bar{c}_1$ or $c_{2} $ is not a real number then $A$ is not $c$-reversible  as characterestic polynomial is not self dual.\\
    Now assuming $\chi_A$ is self dual. (ie $c_3= \bar{c}_1$ and $c_2 \in \mathbf{R}$).\\
    \begin{enumerate}
    \item 
    If the trace sequence \( \{\operatorname{tr}(A^n)\}_{n \in \mathbb{N}} \) is bounded, then \( A \) is either elliptic or parabolic. Since both elliptic and parabolic elements are \( c \)-reversible by corollary \ref{EPL} 
\item 
Now suppose the trace sequence is unbounded, and the minimal polynomial of \( A \) has degree different from 3. Then, up to conjugacy, \( A \) is of one of the following forms:

\paragraph{Case 1: \( \deg(m_A) = 4 \).}
\[
\begin{bmatrix}
re^{i\theta} & 0 & 0 & 0\\
0 & r^{-1}e^{i\theta} & 0 & 0\\
0 & 0 & se^{-i\theta} & 0\\
0 & 0 & 0 & s^{-1}e^{-i\theta}
\end{bmatrix}, \quad
\begin{bmatrix}
re^{i\theta} & 0 & 0 & 0\\
0 & r^{-1}e^{i\theta} & 0 & 0\\
0 & 0 & e^{i\phi} & 0\\
0 & 0 & 0 & e^{-i(2\theta + \phi)}
\end{bmatrix},
\]
\[
\begin{bmatrix}
re^{i\theta} & 1 & 0 & 0\\
0 & re^{i\theta} & 0 & 0\\
0 & 0 & r^{-1}e^{-i\theta} & 1\\
0 & 0 & 0 & r^{-1}e^{-i\theta}
\end{bmatrix}, \quad
\begin{bmatrix}
e^{i\theta} & 1 & 0 & 0\\
0 & e^{i\theta} & 0 & 0\\
0 & 0 & re^{-i\theta} & 0\\
0 & 0 & 0 & r^{-1}e^{-i\theta}
\end{bmatrix}
\]

\paragraph{Case 2: \( \deg(m_A) = 2 \).}
\[
\begin{bmatrix}
re^{i\theta} & 0 & 0 & 0\\
0 & re^{i\theta} & 0 & 0\\
0 & 0 & r^{-1}e^{i\theta} & 0\\
0 & 0 & 0 & r^{-1}e^{i\theta}
\end{bmatrix}, \quad \text{where } e^{4i\theta} = 1.
\]

\paragraph{Case 3: \( \deg(m_A) = 1 \).} 
There is no loxodromic or loxoparabolic element whose minimal polynomial has degree 1.

In all the above cases, the matrices are \( c \)-reversible by Theorem~1.4.
 \item[(3.a)]
So, now  the minimal polynomial of \( A \) has degree 3. Then, up to conjugacy, \( A \) is of the form:
\[
A_1 =
\begin{bmatrix}
r & 1 & 0 & 0\\
0 & r & 0 & 0\\
0 & 0 & r^{-1} & 0\\
0 & 0 & 0 & r^{-1}
\end{bmatrix}, \quad
A_2 =
\begin{bmatrix}
ri & 1 & 0 & 0\\
0 & ri & 0 & 0\\
0 & 0 & r^{-1}i & 0\\
0 & 0 & 0 & r^{-1}i
\end{bmatrix},
\]
\[
A_3 =
\begin{bmatrix}
e^{i\theta} & 0 & 0 & 0\\
0 & e^{i\theta} & 0 & 0\\
0 & 0 & re^{-i\theta} & 0\\
0 & 0 & 0 & r^{-1}e^{-i\theta}
\end{bmatrix}
\]

Observe that \( A_1 \) and \( A_2 \) are not \( c \)-reversible, but \( A_3 \) is. Also note that \( \operatorname{tr}(A_1) \in \mathbb{R} \) and \( \operatorname{tr}(A_2) \in i\mathbb{R} \). So if the trace is of the form \( x + iy \) with \( x \neq 0 \) and \( y \neq 0 \), then \( A \) is \( c \)-reversible. 

 \item[(b)] Now suppose \(\operatorname{tr}(A) \in \mathbb{R} \) or \(\operatorname{tr}(A) \in i\mathbb{R} \)
 \\Now consider the trace of \( A_3 \):
\[
\operatorname{tr}(A_3) = 2\cos\theta + 2i\sin\theta + r\cos\theta - ir\sin\theta + r^{-1}\cos\theta - ir^{-1}\sin\theta
\]
\[
= (2 + r + r^{-1})\cos\theta + i(2 - r - r^{-1})\sin\theta
\]
If \( \operatorname{tr}(A_3) \in \mathbb{R} \), then \( \sin\theta = 0 \) since \( (2 - r - r^{-1}) \neq 0 \). Similarly, if \( \operatorname{tr}(A_3) \in i\mathbb{R} \), then \( \cos\theta = 0 \).

Thus, if \( \operatorname{tr}(A_3) \in \mathbb{R} \), then \( A_3 \) is of the form
\[
A_{3,1} = 
\begin{bmatrix}
1 & 0 & 0 & 0\\
0 & 1 & 0 & 0\\
0 & 0 & r & 0\\
0 & 0 & 0 & r^{-1}
\end{bmatrix}
\quad \text{or} \quad
A_{3,-1} = 
\begin{bmatrix}
-1 & 0 & 0 & 0\\
0 & -1 & 0 & 0\\
0 & 0 & r & 0\\
0 & 0 & 0 & r^{-1}
\end{bmatrix}, \quad |r| \neq 1
\]

Also if \( \operatorname{tr}(A_3) \in i\mathbb{R} \), then \( A_3 \) is of the form
\[
A_{3,i} = 
\begin{bmatrix}
i & 0 & 0 & 0\\
0 & i & 0 & 0\\
0 & 0 & ri & 0\\
0 & 0 & 0 & r^{-1}i
\end{bmatrix}
\quad \text{or} \quad
A_{3,-i} = 
\begin{bmatrix}
-i & 0 & 0 & 0\\
0 & -i & 0 & 0\\
0 & 0 & ri & 0\\
0 & 0 & 0 & r^{-1}i
\end{bmatrix}, \quad |r| \neq 1
\]

Now first we will distinguish \( A_1, A_{3,1}\), and \(A_{3,-1} \).

\begin{itemize}
    \item For \( A_1 \), the second trace coefficient is
    \[
    c_2 = r^2 + 1 + 1 + 1 + 1 + r^{-2} = (r + r^{-1})^2 + 2 = \frac{c_3^2}{4} + 2
    \]
    \item For \( A_{3,1} \),
    \[
    c_2 = 1 + r + r^{-1} + r + r^{-1} + 1 = 2(r + r^{-1} + 2) - 2 = 2c_3 - 2
    \]
    \item For \( A_{3,-1} \),
    \[
    c_2 = 1 - r - r^{-1} - r - r^{-1} + 1 = -2(r + r^{-1} - 2) - 2 = -2c_3 - 2
    \]
\end{itemize}

If \( A_{3,1} \) (respectively \( A_{3,-1} \)) satisfied \( c_2 = \frac{c_3^2}{4} + 2 \), then \( \operatorname{tr}(A_{3,1}) = 4 \) (respectively \( \operatorname{tr}(A_{3,-1}) = -4 \)), which contradicts the condition \( |r| \neq 1 \).

\medskip
Now want to distinguish  \( A_2, A_{3,i}\), and \(A_{3,-i} \).
\begin{itemize}
 \item  For \( A_2 \),
\[
c_2 = -r^2 - 1 - 1 - 1 - 1 - r^{-2} = (ri + r^{-1}i)^2 - 2 = \frac{c_3^2}{4} - 2
\]

 \item  For \( A_{3,i} \),
\[
c_2 = -1 - r - r^{-1} - r - r^{-1} - 1 = 2i(ri + r^{-1}i + 2i) + 2 = 2ic_3 + 2
\]

 \item  For \( A_{3,-i} \),
\[
c_2 = -1 + r + r^{-1} + r + r^{-1} - 1 = -2i(ri + r^{-1}i - 2i) + 2 = -2ic_3 + 2
\]

 Similarly, if \( A_{3,i} \) (respectively \( A_{3,-i} \)) satisfied \( c_2 = \frac{c_3^2}{4} - 2 \), then \( \operatorname{tr}(A_{3,i}) = 4i \) (respectively \( \operatorname{tr}(A_{3,-i}) = -4i \)), which is also not possible. 
\end{itemize} \end{enumerate}
This completes the proof. \qed
        
\section*{Declarations} 

	{\bf Ethical Approval}. not applicable.

{\bf Competing interests}. not applicable.

	{\bf Funding}. not applicable.

{\bf Authors' Contributions.} All authors have contributed equally. 

{\bf Availability of data and materials}. not applicable.

\end{document}